\documentclass[11pt,reqno,a4paper]{amsart} %(principal)
\usepackage{graphicx}
\usepackage{amssymb,amsmath}
\usepackage{amsthm}
\usepackage{color,graphicx}
\usepackage{hyperref}
\usepackage{color}
\usepackage{mathabx}
\usepackage{epstopdf}

\usepackage{subfigure} 

\usepackage{verbatim}

\newcommand{\be}{\begin{equation}}
	
	\newcommand{\ee}{\end{equation}}

\newtheorem{prop}{Proposition}[section]

\newtheorem{obs}{Remark}[section]
\newtheorem{coro}{Corollary}[section]

\numberwithin{equation}{section}
\numberwithin{figure}{section}

\newtheorem{theorem}{Theorem}[section]
\newtheorem{proposition}[theorem]{Proposition}
\newtheorem{remark}[theorem]{Remark}
\newtheorem{lemma}[theorem]{Lemma}

\newtheorem{definition}[theorem]{Definition}

\begin{document}
	%\color{red}
	\vglue-1cm \hskip1cm
	\title[Transversal instability of periodic traveling waves]{Transversal spectral instability of periodic traveling waves for the generalized  Zakharov-Kuznetsov equation
	}

	\begin{center}
		
		

		\maketitle

		{\bf F\'abio Natali}
		
		{Departamento de Matem\'atica - Universidade Estadual de Maring\'a\\
			Avenida Colombo, 5790, CEP 87020-900, Maring\'a, PR, Brazil.}\\
		{ fmanatali@uem.br}

		\vspace{3mm}

	\end{center}
	
	\begin{abstract}
	In this paper, we determine the transversal instability of periodic traveling wave
	solutions of the generalized Zakharov-Kuznetsov equation in two space dimensions. Using an adaptation of the arguments in \cite{nikolay} in the periodic context, it is possible to prove that all positive and one-dimensional $L-$periodic waves are spectrally (transversally) unstable. In addition, when periodic waves that change their sign exist, we also obtain the same property when the associated projection operator defined in the zero mean Sobolev space has only one negative eigenvalue. 
		
	\end{abstract}

	\section{Introduction} 
	
In this paper, we consider the generalized Zahharov-Kuznetsov equation (gZK henceforth)
	\begin{equation}u_t+u^pu_x+(\Delta u)_x=0\label{ZK}\end{equation}
	posed on $\mathbb{T}_L\times \mathbb{R}$, that is, the evolution $u=u(x,y,t)$ is a real-valued function and it is defined in $\mathbb{T}_L\times\mathbb{R}\times\mathbb{R}$. Here, $p>0$ and the set $\mathbb{T}_L$ indicates the $L-$torus. All functions defined on it may be seen as periodic functions on the real line with period $L>0$.\\
	\indent Let us consider periodic traveling waves propagating only in the first variable and with wave speed $c>0$, that is, suppose that $u(x,y,t)=\varphi(x-ct)$ is a solution of $(\ref{ZK})$. Substituting this form into the equation $(\ref{ZK})$, we obtain after integration the following ODE
	\begin{equation}\label{ode1}-\varphi''+c\varphi-\frac{1}{p+1}\varphi^{p+1}+A=0,\end{equation}
	where $A$ is a constant of integration which we will assume zero, that is, $A\equiv0$.\\
	\indent In what follows, we consider the perturbation of the evolution $u(x,y,t)$ associated  to the equation $(\ref{ZK})$ of the form 
	\begin{equation}\label{pertub}
	u(x,y,t)=v(x-ct,y,t)-\varphi(x-ct).
	\end{equation}
After some computations and neglecting the nonlinear terms, we obtain from the equation $(\ref{ode1})$, that $v$ satisfies the following linear equation
\begin{equation}\label{lineareq1}
	v_t=\partial_x(\mathcal{L}-\partial_y^2)v,
\end{equation}
where $\mathcal{L}$ is the linearized operator given by	\begin{equation}\label{operator}\mathcal{L}=-\partial_x^2+c-\varphi^p.\end{equation}
	Suppose that equation $(\ref{lineareq1})$ admits a growing mode solution of the form $v(x,y,t)=e^{\lambda t}e^{iky}w(x)$, where $w$ is an $L-$periodic smooth function. Substituting this form of solution into the equation $(\ref{lineareq1})$, we obtain the following spectral problem
	\begin{equation}\label{specprob1}
		\partial_x(\mathcal{L}+k^2I)w=\lambda w.
		\end{equation}
	\indent Problem $(\ref{specprob1})$ can be seen in an equivalent form as
	\begin{equation}\label{specprobmean}
		\partial_x(Q\mathcal{L}+k^2I)w=\lambda w,
	\end{equation}
	where $Q\mathcal{L}$ is the projection of the operator $\mathcal{L}$ in the space $L_{per,m}^2([0,L])$ constituted by periodic (classes of) functions in $L_{per}^2([0,L])$ with the zero mean property. $Q\mathcal{L}$ is then defined as
	\begin{equation}\label{QL}Q\mathcal{L}=\mathcal{L}+\frac{1}{L}(\varphi^p,\cdot)_{L_{per}^2}.\end{equation}
	For $\lambda\neq0$, it is important to notice that we are forced to consider, because of the problem $(\ref{specprobmean})$, that $w$ has the zero mean property just by integrating both sides of the equality $(\ref{specprobmean})$ over the interval $[0,L]$. In addition, since the linear operator $\partial_x: H_{per,m}^1([0,L])\rightarrow L_{per,m}^2([0,L])$ is invertible with  bounded inverse $\partial_x^{-1}:L_{per,m}^2([0,L])\rightarrow H_{per,m}^1([0,L])$, we can also consider the unique $U\in L_{per,m}^2([0,L])$ such that $w=\partial_x^{-1}U$. This fact enables us to consider the new spectral (and correct) problem in the periodic context given by,
		\begin{equation}\label{specprobmean1}
		(\partial_xQ\mathcal{L}\partial_x^{-1}+k^2I)U=\lambda \partial_x^{-1} U,
	\end{equation}  	
	\indent Therefore, the problem of the transversal stability reduces to a spectral stability problem where the propagation of the wave is considered just in one direction. More specifically, we have:
	\begin{definition}\label{defistab1}
		The periodic wave $\varphi \in H^{2}_{per}([0,L])$ is said to be transversally spectrally stable if $\sigma(\partial_xQ\mathcal{L}\partial_x^{-1}+k^2I) \subset i\mathbb{R}$ in  $L_{per,m}^2([0,L])$ for all $k>0$. Otherwise, that is, if $\sigma(\partial_xQ\mathcal{L}\partial_x^{-1}+k^2I)$ in $L_{per,m}^2([0,L])$ contains a point $\lambda$ with $Re(\lambda)>0$ for some $k>0$, the periodic wave $\varphi$ is said to be transversally spectrally unstable.
	\end{definition}

\indent As far as we can see, the transverse instability of traveling waves associated to the equation $(\ref{ZK})$ has
been determined for both periodic and solitary waves cases. For the case $\mathbb{R}\times\mathbb{R}$, we see that the author in \cite{bridges} determined a geometric condition for the long wavelength transverse instability of solitary solutions by using the multi-symplectic structure of the equation $(\ref{ZK})$. Using the explicit form of the solution for the equation $(\ref{ode1})$ with hyperbolic secant profile (that is, $A=0$ in equation $(\ref{ode1})$), the author derived an index which could be viewed as the Jacobian of a particular map and whose sign determined the transverse stability of the underlying wave. In \cite{johnson}, the author derived sufficient conditions for the transverse instability in case $\mathbb{T}_L\times \mathbb{R}$ by using the well-known result of spectral stability of periodic waves associated to the Korteweg-de Vries equation. The approach then was effectively used (by considering small values of $A$ in $(\ref{ode})$) to conclude the transversal instability of periodic waves with large periods or waves located in a neighbourhood of the stationary solution, that is, when the periodic wave is close up to the solitary wave or the equilibrium solution in the associated phase portrait corresponding to the equation $(\ref{ode})$.\\
\indent In \cite{yamazaki}, the author studied the case $\mathbb{R}\times\mathbb{T}_L$ and $p=1$ in equation $(\ref{ZK})$. Using the arguments in \cite{pego}, it is possible to conclude the transversal stability for $L\in \left(0,\frac{2}{\sqrt{5c}}\right]$ and the transversal instability when $L> \frac{2}{\sqrt{5c}}$. Here, the value $\frac{2}{\sqrt{5c}}$ is associated with the unique negative eigenvalue of the linearized operator $\mathcal{L}=-\partial_x^2+c-2Q_c$, where $Q_c$ is the solitary wave with hyperbolic secant profile. More specifically, we have $\lambda_c=-\frac{5c}{4}$, where $\mathcal{L}\chi=\lambda_c\chi$ for some smooth non-zero periodic function $\chi$. In addition, the author establishes orbital (asymptotic) stability and instability results using the transversal stability and instability previously obtained and jointly with a suitable global well-posedeness results.\\
	\indent The methods presented in our contribution are based on a adaptation of the approach in \cite{nikolay} which established a simple criterion for the transversal instability for solitary waves. The spectral problem in $(\ref{specprobmean1})$ fits in some sense in the framework in \cite{nikolay}, but the main problem is that our operator $\partial_xQ\mathcal{L}\partial_x^{-1}+k^2I$ in $(\ref{specprobmean1})$ is not self-adjoint as requested in \cite{nikolay}. To overcome this difficulty, we prove some facts of spectral theory to guarantee that $\partial_xQ\mathcal{L}\partial_x^{-1}+k^2I$ and the self-adjoint operator $Q\mathcal{L}+k^2I$ have the same spectra. Using some parts of the proof in \cite[Theorem 1.1]{nikolay}, it is possible to prove the existence of $k_0>0$ such that $\dim(\ker(Q\mathcal{L}+k_0^2I))=1$ and since both subspaces $\ker((Q\mathcal{L}+k_0^2I))$ and $\ker(\partial_xQ\mathcal{L}\partial_x^{-1}+k_0^2I)$ have the same dimension (see Lemma $\ref{dimkerlema}$), we obtain that $\dim(\ker(\partial_xQ\mathcal{L}\partial_x^{-1}+k_0^2I))=\dim(\ker((\partial_xQ\mathcal{L}\partial_x^{-1}+k_0^2I)^*))=1.$ The fact $0$ is an isolated eigenvalue of $\partial_xQ\mathcal{L}\partial_x^{-1}+k_0^2I$ enables us to conclude that ${\rm range}(\partial_xQ\mathcal{L}\partial_x^{-1}+k_0^2I)$ is closed and thus, 
	$${\rm codim}({\rm range} (\partial_xQ\mathcal{L}\partial_x^{-1}+k_0^2I))=\dim ({\rm ker}(\partial_xQ\mathcal{L}\partial_x^{-1}+k_0^2I))=1,$$
	that is, $\partial_xQ\mathcal{L}\partial_x^{-1}+k_0^2I$ is a Fredholm operator with zero index. This last fact is suitable to prove the transversal spectral instability of $\varphi$ by using the Lyapunov-Schmidt reduction  to study the eigenvalue problem $(\ref{specprobmean1})$ in the vicinity of $\nu = 0$, $k = k_0$ and $U=\psi$, where $\psi$ is in the kernel of $\partial_xQ\mathcal{L}\partial_x^{-1}+k_0^2I$ and such that $||\psi||_{L_{per}^2} = 1$.
	
\indent Our result is now established. In order to simplify the notation, we define for $k\geq0$ the following linear operators
	\begin{equation}\label{rk}R(k)=Q\mathcal{L}+k^2I\end{equation}
and
\begin{equation}P(k)=\partial_xQ\mathcal{L}\partial_x^{-1}+k^2I.
	\label{pk}\end{equation}
	
	\begin{theorem}\label{mainT}
		 Let $\varphi$ be a positive and periodic solution for the equation $(\ref{ode1})$ with $A=0$. There exist $\nu>0$, $k\neq0$ and $U\in H_{per,m}^2([0,L])\backslash\{0\}$ such that the spectral problem
		\begin{equation}\label{specprob2}P(k)U=\nu \partial_x^{-1}U,
			\end{equation}
		is verified. In particular, the periodic traveling wave solution $\varphi$ is transversally (spectrally) unstable. Moreover, let $p$ be a positive even number and consider $\varphi$ a periodic solution that changes its sign for the equation $(\ref{ode1})$ with $A=0$. The periodic solution $\varphi$ is transversally unstable provided that $R(0)$ has only one negative eigenvalue.
	\end{theorem}

\begin{remark}
	It is important to mention that the proof of Theorem $\ref{mainT}$ is similar to the proof of \cite[Theorem 1.1]{nikolay} but adapted to the periodic context. As we have mentioned above, our operator $P(k)$ in $(\ref{pk})$ does not satisfy the self-adjointness requirement as specified in \cite{nikolay}. To overcome this difficulty, we must establish a suitable characterization of the spectrum of $P(k)$ in terms of the self-adjoint operator $R(k)$ in $(\ref{rk})$. Additionally, we need to find a non-zero value of $k_0$ such that $P(k_0)$ is a Fredholm operator with zero index. This last condition is crucial for applying the Lyapunov-Schmidt reduction and demonstrating the existence of $\nu>0$, $k\neq0$, and a non-zero $U\in H_{per,m}^2([0,L])$ such that the spectral problem $(\ref{specprob2})$ is verified. To avoid plagiarism, we will omit the portions of our work that coincide with the proof of the main result in \cite{nikolay}. 
\end{remark}
\indent Finally, we present some comments concerning the nonlinear instability of the periodic wave $\varphi$ (see Section 4) that is transversally (spectrally) unstable according to the Theorem $\ref{mainT}$. To do so, we need to use the general setting in \cite{henry} (see also \cite{nikolay1}, \cite{nikolay2} and \cite{nikolay3} for additional references). It is important to mention that the nonlinearity present in the equation $(\ref{ZK})$ is of the form $u^pu_x$, where $p>0$. As far as we know, to prove the nonlinear instability, it is suitable to obtain a convenient global well-posedness result in the energy space $X=H_{per}^1(\mathbb{T}_L\times\mathbb{R})$ for the Cauchy problem associated to the evolution equation. For most values of $p>0$, we do not have any information about the (local) well-posedness result in the energy space $X$, but global solutions in time are expected if $1\leq p <2$ by using the Gagliardo-Nirenberg inequality applied to the conservation law $E$ in $(\ref{conservada1})$. Regarding our study of spectral stability, we do not need a convenient well-posedness result since we can consider smooth solutions $u$ of the form $(\ref{pertub})$. In fact, $\varphi$ is smooth and $v$ is decomposed as a product of exponentials factors in time $e^{\lambda t}$ and in the spatial variable $e^{iky}$, with a smooth solution $w$ depending on $x$. Function $w$ can be considered smooth because of a bootstrapping argument applied to the equation $(\ref{specprob2})$.\\
	\indent Our paper is organized as follows: Section 2 is devoted to present the existence of periodic waves via planar analysis. In Section 3, we prove Theorem $\ref{mainT}$ and present some concrete examples. In Section 4, we present some remarks concerning the nonlinear instability. 
	\section{Existence of periodic solutions via planar analysis.}
	In this section, we present some basic facts concerning the existence of
	periodic solutions for the nonlinear equation $(\ref{ode1})$ for $A=0$, that is,
	\begin{equation}
	-\varphi''+  c\varphi -\frac{1}{p+1}\varphi^{p+1} = 0, \label{ode}
	\end{equation}
	where $c>0$ and $p>0$ are real numbers. 
	
	It is well known that quation (\ref{ode}) is
	conservative and the eventual periodic solutions are contained on the level
	curves of the energy
	\begin{equation}\label{energyODE}
		\mathcal{E}(\varphi, \xi) =  \frac{\xi^2}{2} -\frac{c\varphi^2}{2}+\frac{\varphi^{p+2}}{(p+1)(p+2)},
	\end{equation}
	where $\xi=\varphi'$.\\
	\indent By classical theories of ordinary differential equations (see \cite{chicone} for further details), we see that $\varphi$ is a periodic solution of the equation $(\ref{ode})$ if, and only if, $(\varphi,\varphi')$ is a periodic orbit of the planar differential system
	\begin{equation}\label{planarODE}
		\left\{\begin{array}{lllll}
			\varphi'=\xi,\\\\
			\xi'=c\varphi-\frac{1}{p+1}\varphi^{p+1}.
		\end{array}\right.
	\end{equation}
	\indent The periodic orbits for the equation $(\ref{planarODE})$ can be determined by considering the energy levels of the function $\mathcal{E}$ defined in $(\ref{energyODE})$. This means that the pair $(\varphi,\xi)$ satisfies the equation $\mathcal{E}(\varphi,\xi)=B$. If $p>0$ and $B\in \left(B_0,0\right)$, we obtain periodic orbits which turn round at the equilibrium points $(((p+1)c)^{1/p},0)$. Here, $B_0$ is a negative number defined as $B_0=-\frac{p(p+1)^{\frac{2}{p}}c^{\frac{p+2}{p}}}{2(p+2)}$. In our specific case, we see that $(\ref{planarODE})$ has at least two critical points, being one saddle point at $(\varphi,\xi)=(0,0)$ and one center point at $(\varphi,\xi)=(((p+1)c)^{1/p} ,0)$. According to the standard theories of ordinary differential equations, the periodic orbits emanate from the center points to the separatrix curve which is represented by a smooth solution $\widetilde{\varphi}:\mathbb{R}\rightarrow\mathbb{R}$ of $(\ref{ode})$ satisfying $\lim_{x\rightarrow \pm \infty}\widetilde{\varphi}^{(n)}(x)=0$ for all $n\in\mathbb{\mathbb{N}}$. When $p$ is in particular an even integer, we see that the presence of two symmetric center points $(\pm ((p+1)c)^{1/p},0)$ allows to conclude that the periodic orbits which turn around these points can be negative and positive. Outside the separatrix, we have the existence of periodic solutions that change their sign. Indeed, if $B>0$ we also have periodic orbits and the corresponding periodic solutions  $\varphi$ with the zero mean property, that is, periodic solutions satisfying  $\int_0^L\varphi(x)dx=0$. Independently of the type of periodic solutions which we are working on, the period $L=L(B)$ of the solution $\varphi$ can be expressed (formally) by
	
	\begin{equation}
	L=\displaystyle 2\int_{b_1}^{b_2}\frac{dh}{\sqrt{-\frac{2h^{p+2}}{(p+1)(p+2)}+ ch^2+2B}},
	\label{persol}\end{equation}
	where $b_1=\displaystyle\min_{x\in [0,L]}\varphi(x)$ and $b_2=\displaystyle\max_{x\in [0,L]}\varphi(x)$.\\
	\indent On the other hand, the energy levels of the first integral $\mathcal{E}$ in $(\ref{energyODE})$ parametrize the unbounded set of periodic orbits $\{\Gamma_B\}_B$ which emanate from the separatrix curve. Thus, we can conclude that the set of smooth periodic solutions of $(\ref{ode})$ can be expressed by a smooth family $\varphi=\varphi_{B}$ which is parametrized by the value $B$. %For each $B\in (B_{0},0)\cup(0,+\infty)$ fixed, we also obtain that the orbit $\Gamma_B$ is equal to the corresponding periodic solution $\varphi$ and the period of the orbit is given by the smooth map in terms of $B\in (B_{0},0)\cup(0,+\infty)$
	%\be
	%\widetilde{L}=\int_{\Gamma_B}\frac{dh}{\chi}.
	%\label{perorbit}
	%\ee
	%We see from $(\ref{persol})$ and $(\ref{perorbit})$ that $L=\widetilde{L}$. 
	Moreover, when $B\in \left(B_0,0\right)$, we see that if $B\rightarrow B_0$, we have $L\rightarrow  \alpha(c)>0$ (stationary solution), and if $B\rightarrow0$, one has $L\rightarrow +\infty$ (solitary wave solution). On the other hand, when $B\in (0,+\infty)$, we see that if $B\rightarrow 0$, we have $L\rightarrow +\infty$ (solitary wave solution), and if $B\rightarrow+\infty$, we obtain $L\rightarrow 0$.

	\section{Transversal spectral instability of periodic waves for gZK.}
	\indent Before proving our main result, we need some basic facts concerning spectral theory. The first result is now given and establishes a similarity regarding the resolvent set of the operators $R(k)$ and $P(k)$ for all $k\geq0$.
	\begin{lemma}\label{lemaresolvente} Let $R(k)$ and $P(k)$ be the linear operators defined in $(\ref{rk})$ and $(\ref{pk})$, respectively. For all $k\geq0$, we have $\rho(R(k))=\rho(P(k))$, where $\rho(\mathcal{A})$ indicates the resolvent set of certain linear operator $\mathcal{A}$. 
	\begin{proof} It suffices to prove the result for the case $k=0$ since $\partial_xQ\mathcal{L}\partial_x^{-1}+k^2I=\partial_x(Q\mathcal{L}+k^2I)\partial_x^{-1}$. Let $\lambda\in \rho(P(0))$ be fixed. Thus $P(0)-\lambda I$ is invertible and for each $v\in L_{per,m}^2([0,L])$, there exists a unique $u\in H_{per,m}^2([0,L])$ such that $(P(0)-\lambda I)u=v$. On the other hand, since $\partial_x:H_{per,m}^1([0,L])\rightarrow L_{per,m}^2([0,L])$ is also invertible, there exists a unique $w\in H_{per,m}^1([0,L])$ such that $v=\partial_x w$. Since $w=\partial_x^{-1}v$, we obtain, from the fact $(Q\mathcal{L}-\lambda I)\partial_x^{-1}u=\partial_x^{-1}v$ that $(Q\mathcal{L}-\lambda I)r=w,$ where $r=\partial_x^{-1}u$. Then, we deduce $Q\mathcal{L}-\lambda I$ is invertible and since $Q\mathcal{L}-\lambda I$ is a self-adjoint closed operator with ${\rm ker} (Q\mathcal{L}-\lambda I)=\{0\}$ and ${\rm range}(Q\mathcal{L}-\lambda I)=L_{per,m}^2([0,L])$, we have from the closed graph theorem $\lambda \in \rho (R(0))$, that is, $\rho(P(0))\subset \rho(R(0)$.\\
	\indent Next, let $\lambda\in \rho(R(0))$ be fixed. Since $Q\mathcal{L}-\lambda I$ is invertible and $\partial_xQ\mathcal{L}\partial_x^{-1}-\lambda I=\partial_x(Q\mathcal{L}-\lambda I)\partial_x^{-1}$, we see that $\partial_xQ\mathcal{L}\partial_x^{-1}-\lambda I$ is invertible and 
	$(\partial_xQ\mathcal{L}\partial_x^{-1}-\lambda I)^{-1}=\partial_x(Q\mathcal{L}-\lambda I)^{-1}\partial_x^{-1}$. On the other hand, the adjoint operator of $P(0)$ is given by $P(0)^{*}=\partial_x^{-1}Q\mathcal{L}\partial_x$ and it is defined in $L_{per,m}^2([0,L])$ with dense domain $H_{per,m}^2([0,L])$. Since $(\partial_x^{-1}Q\mathcal{L}\partial_x)^{*}=\partial_xQ\mathcal{L}\partial_x^{-1}$ is a closed operator with ${\rm ker}(\partial_xQ\mathcal{L}\partial_x^{-1}-\lambda I)=\{0\}$ and ${\rm range}(\partial_xQ\mathcal{L}\partial_x^{-1}-\lambda I)=L_{per,m}^2([0,L])$, we obtain again by the closed graph theorem that $\lambda\in \rho(P(0))$, so that $\rho(R(0))\subset \rho(P(0))$. The lemma is now concluded.
		\end{proof}
	
		\end{lemma}
	
\begin{remark}\label{remessespec} Lemma $\ref{lemaresolvente}$ enables us to conclude that $\sigma(P(k))=\sigma(R(k))$ for all $k\geq0$, where $\sigma(\mathcal{A})$ denotes the spectrum set of a certain linear operator $\mathcal{A}$. Since $\sigma(R(k))$ is also constituted only by a discrete set of eigenvalues accumulating at the infinity, we obtain that $\sigma(P(k))$ is constituted only by discrete eigenvalues with the same behaviour. As a consequence of this fact is that there is no essential spectrum associated to the operator $P(k)$.
\end{remark}	
	
\begin{lemma}\label{closedrange}
Suppose that $0\in \sigma(R(k_0))$ for some $k_0>0$. We have that ${\rm range}(P(k_0))$ is a closed set in $L_{per,m}^2([0,L])$.
\end{lemma}
	\begin{proof}
	We see that $P(k_0)$ is a closed linear operator defined in the Hilbert space $L_{per,m}^2([0,L])$ with dense domain $H_{per,m}^2([0,L])$. Thus, it suffices to prove that $0$ is not an accumulation point of the spectrum $\sigma(P(k_0)^{*}P(k_0))$ of $P(k_0)^{*}P(k_0)$. In fact, we see by hypothesis of the lemma and Remark $\ref{remessespec}$ that $0$ is an isolated eigenvalue of the linear operator $P(k_0)$. Since $P(k_0)$ is a closed linear operator, we obtain $P(k_0)^{**}=P(k_0)$, so that $P(k_0)^{*}P(k_0)$ is self-adjoint. Because of the compact embeddings $H_{per,m}^2([0,L])\hookrightarrow H_{per,m}^1([0,L])\hookrightarrow L_{per,m}^2([0,L])$, we have for $\mu>0$ large enough that the operator $P(k_0)^*P(k_0)+\mu I$ is invertible with bounded inverse. Using the spectral theorem for compact and self-adjoint operators, it follows that the spectrum $\sigma(P(k_0)^{*}P(k_0)+\mu I)$ of $P(k_0)^*P(k_0)+\mu I$ is constituted only by a discrete set of eigenvalues, so that the same behaviour occurs for the spectrum $\sigma(P(k_0)^{*}P(k_0))$ of $P(k_0)^*P(k_0)$. Thus $0$ is an isolated point of the spectrum $\sigma(P(k_0)^*P(k_0))$ of $P(k_0)^*P(k_0)$ and ${\rm range}(P(k_0))$ is closed as requested.
	\end{proof}
	
\begin{remark}\label{remfred} Lemma $\ref{closedrange}$ guarantees that if $0\in \sigma(R(k_0))$ for some $k_0>0$, we obtain ${\rm range}(P(k_0))$ is a closed subspace of $L_{per,m}^2([0,L])$ and therefore, $${\rm codim}({\rm range} (P(k_0)))=\dim ({\rm ker}(P(k_0)^*)).$$ 

\end{remark}

\begin{lemma}\label{dimkerlema}
For all $k\geq0$, we have that $\dim ({\rm ker}(P(k)^*))=\dim ({\rm ker}(P(k)))=\dim({\rm ker}(R(k)))$.

\end{lemma}	

\begin{proof} By spectral theorem for compact and self-adjoint operators, it follows that $\dim({\rm ker}(R(k)))$ is finite. Let us consider $\{v_1,v_2,\cdots,v_n\}$ a basis for ${\rm ker}(R(k))$. We see that $\{\partial_x v_1,\partial_x v_2,\cdots,\partial_x v_n\}$ is a basis for ${\rm ker}(P(k))$ while $\{\partial_x^{-1} v_1,\partial_x^{-1} v_2,\cdots,\partial_x^{-1} v_n\}$ is a basis for ${\rm ker}(P(k)^*)$. In fact, we verify only the first claim since the second one is similar. Let $w$ be an element in ${\rm ker}(P(k))$. Since $(\partial_xQ\mathcal{L}\partial_x^{-1}+k^2I)w=\partial_x(Q\mathcal{L}+k^2)\partial_x^{-1}w=0$, we see that $(Q\mathcal{L}+k^2)\partial_x^{-1}w=0$ and $\partial_x^{-1}w$ belongs to ${\rm ker}(R(k))$. Consequently, $w$ can be written uniquely as a linear combination of the elements $\partial_x v_1,\partial_x v_2,\cdots,\partial_x v_n$. Thus, we get that $\{\partial_x v_1,\partial_x v_2,\cdots,\partial_x v_n\}$ determines a basis for ${\rm ker}(P(k))$ and the proof of the lemma is now completed.
\end{proof}

\begin{proposition}\label{propR0}
	The first eigenvalue of $R(0)$ defined in $(\ref{rk})$ is negative, simple and it can be considered an even periodic function.
\end{proposition}
\begin{proof}
	Let $k_1>0$ be a fixed positive integer (large enough) and consider the associated operator $R(k_1)=Q\mathcal{L}+k_1^2I$. Define the basic set $$\mathcal{C}=\left\{u\in C_{per,m,e}^{\infty}([0,L]);\ u\geq0\ \mbox{in}\ \left[0,\frac{L}{4}\right]\right\},$$
	where $C_{per,m,e}^{\infty}([0,L])$ denotes the space constituted by smooth even $L-$periodic functions with the zero mean property. We see that $\mathcal{C}$ is a cone since for all $\beta\geq0$ and $u\in\mathcal{C}$, we have $\beta u\in \mathcal{C}$ and $\mathcal{C}\cap (-\mathcal{C})=\{0\}$, where $-\mathcal{C}$ is defined as $-\mathcal{C}=\left\{u\in C_{per,m,e}^{\infty}([0,L]);\ u\leq0\ \mbox{in}\ \left[0,\frac{L}{4}\right]\right\}$. The interior set of $\mathcal{C}$, denoted by $\mathcal{C}^o$, is clearly non-empty and then, $\mathcal{C}$ is called a solid cone. For $k_1>0$ large enough, we have that $R(k_1)^{-1}$ is defined in $L_{per,m,e}^2([0,L])$, it has bounded inverse and   $R(k_1)^{-1}:L_{per,m,e}^2([0,L])\rightarrow L_{per,m,e}^2([0,L])$ is a compact operator. Suppose that $v\in C_{per,m,e}^{\infty}([0,L])$ satisfies $v>0$ in $\left[0,\frac{L}{4}\right]$. There is a unique $u\in C_{per,m,e}^{\infty}([0,L])$ such that $v=R(k_1)u$. By an application of the maximum principle in the one-dimensional case (see \cite[Chapter 1, Theorem 23]{protter}), we obtain $u>0$ in $\left[0,\frac{L}{4}\right]$, that is, $R(k_1)^{-1}$ is totally positive. Using the standard Krein-Rutman Theorem, we obtain the existence of $w\in \mathcal{C}^o$ such that $(Q\mathcal{L}+k_1^2I)^{-1}w=\varrho w,$ where $\varrho>0$ is the spectral radius of the operator $(Q\mathcal{L}+k_1^2I)^{-1}$ with $\varrho$ being a simple isolated eigenvalue. Since $k_1>0$ is large enough, the first 
	eigenvalue of $Q\mathcal{L}$ is negative and simple.
\end{proof}

\begin{proposition}\label{propLR}
Let $\mathcal{L}$ be the operator defined in $(\ref{operator})$ and suppose that $\mathcal{L}$ has only one negative eigenvalue. Thus $R(0)=Q\mathcal{L}$ defined in $(\ref{rk})$ has only one negative eigenvalue.
\end{proposition}
\begin{proof}
According with \cite[Theorem 1.1]{alvesnatali} we see that the kernel of $\mathcal{L}$ is simple and generated by $\varphi'$. In addition, the case where $\mathcal{L}$ has only one negative eigenvalue occurs only when $\varphi$ is an even positive and periodic solution associated to the equation $(\ref{ode})$ (see Theorem 1.1 and Lemma 3.1 in \cite{alvesnatali}). Since the first eigenvalue of $R(0)$ is negative, we see that $n(R(0))\geq1$, where $n(\mathcal{A})$ denotes the quantity of negative eigenvalues of a certain linear operator $\mathcal{A}$. The fact that kernel of $\mathcal{L}$ is simple enables us to conclude, by Index Theorem \cite[Theorem 5.3.2]{KP}, that $n(R(0))=n(\mathcal{L})-n_0-z_0$, where $z_0$ and $n_0$ are non-negative integers which are related with the quantity $(\mathcal{L}^{-1}1,1)_{L_{per}^2}$. In fact, if $(\mathcal{L}^{-1}1,1)_{L_{per}^2}<0$, we have $n_0=1$ and $z_0=0$, while $(\mathcal{L}^{-1}1,1)_{L_{per}^2}>0$ implies $n_0=z_0=0$ and for $(\mathcal{L}^{-1}1,1)_{L_{per}^2}=0$, we obtain $n_0=0$ and $z_0=1$. Since $n(R(0))\geq1$ and $n(\mathcal{L})=1$, we deduce $z_0=n_0=0$, so that $n(R(0))=1$ as requested.
\end{proof}

\noindent \textit{Proof of the Theorem $\ref{mainT}$.} The proof of this result is similar to \cite[Theorem 1.1]{nikolay} and because of this, we only give the main steps. We have to notice that if $n(R(0))=1$ gives us exactly assumption ${\rm (H4)}$ in \cite{nikolay}.  First, according with Remark $\ref{remessespec}$, we see that $\sigma(R(k))$ and $\sigma(P(k))$ are constituted only by a discrete set of eigenvalues of $R(k)$ and $P(k)$, respectively. Second, in terms of the inner product of $L_{per}^2([0,L])$, we see that if $k_1\geq k_2\geq0$, we have $R(k_1)\geq R(k_2)$. In addition, suppose the existence of $k>0$ and $U\neq0$, such that $R(k)U=0$. Clearly, we have $(R'(k)U,U)_{L_{per}^2}=2k(U,U)_{L_{per}^2}>0$. The two last facts are exactly assumptions ${\rm (H2)}$ and ${\rm (H3)}$ of \cite{nikolay} concerning the linear and self-adjoint operator $R(k)$.  Next, we establish the existence of $k_0>0$ such that $\ker(R(k_0))$ is one-dimensional. In fact, we define 
\begin{equation}\label{functionfk}
	f(k)=\inf_{||u||_{L_{per}^2}=1}(R(k)u,u)_{L_{per}^2}.
	\end{equation}
In any case, that is, if $\varphi$ positive\footnote{For positive and periodic solutions see Proposition $\ref{propLR}$} or a periodic solution that changes its for $(\ref{ode})$, we have that $R(0)$ has only one negative eigenvalue which is simple. Thus, there exists a unique $\beta>0$ and an associated $\chi_{\beta}\in H_{per,m}^2([0,L])\backslash\{0\}$ such that $R(0)\chi_{\beta}=Q\mathcal{L}\chi_{\beta}=-\beta \chi_{\beta}$ and $||\chi_{\beta}||_{L_{per}^2}=1$. Thus 
$$f(0)=\inf_{||u||_{L_{per}^2}=1}(R(0)u,u)_{L_{per}^2}\leq (-\beta \chi_{\beta},\chi_{\beta})_{L_{per}^2}=-\beta<0.$$
\indent On the other hand, since $R(0)$ has only one negative eigenvalue, we obtain for a large $k>0$ that $R(k)=R(0)+k^2I$ is a positive operator, so that $f(k)>0$ for large values of $k>0$. The intermediate value theorem then implies the existence of $k_0>0$ such that $f(k_0)=0$, where $f(k)<0$ for all $k\in[0,k_0)$. Using the same arguments as in \cite[Theorem 1.1]{nikolay}, we obtain that $\dim(\ker(R(k_0)))=1$.\\
\indent In addition, since $R(k_0)$ is a self-adjoint closed operator, we can use again Lemma $\ref{closedrange}$ to conclude that ${\rm range}(R(k_0))$ is a closed subspace contained in $L_{per,m}^2([0,L])$. Thus, it follows that
\begin{equation}\label{codimrk}{\rm codim}({\rm range} (R(k_0)))=\dim ({\rm ker}(R(k_0)))=1.\end{equation}
Thus, by Remark $\ref{remfred}$, Lemma $\ref{dimkerlema}$, and $(\ref{codimrk})$, we have 
\begin{equation}\label{codimrk1}{\rm codim}({\rm range} (P(k_0)))=\dim ({\rm ker}(P(k_0)^*))=\dim ({\rm ker}(P(k_0)))=1.\end{equation}
\indent To finish, we see by Remark $\ref{remessespec}$ that $\sigma(P(k_0))$ is constituted by a discrete set of isolated eigenvalues and by $(\ref{codimrk1})$ it follows that $P(k_0)$ is a Fredholm operator with zero index. We then use the Lyapunov-Schmidt method to study the eigenvalue problem $(\ref{specprob2})$ in the vicinity of $\nu = 0$, $k = k_0$ and $U=\psi$, where $\psi$ is in the kernel of $P(k_0)$ and such that $||\psi||_{L_{per}^2} = 1$. In fact, we want to find $W=\psi+V$, where $V\in\{\psi\}^{\bot}$ and our intention is to solve $G(V,k,\nu)=0$ with $\nu>0$, where
$$G(V,k,\nu)=P(k)\psi+P(k)V-\nu\partial_x^{-1}\psi-\nu\partial_x^{-1} V,\ \ \ \ V\in\{\psi\}^{\bot}.$$
The reminder of the proof is similar as the final of the proof of \cite[Theorem 1.1]{nikolay} and because of this, we omit the details.
\subsection{Examples} We present some examples to show the effectiveness of our result.\\\\
\noindent \textbf{a) Positive periodic waves.} Let $p>0$ be fixed. According with facts presented in Section 2, we obtain positive and periodic waves with turns around the equilibrium point $(((p+1)c)^{1/p},0)$ in the phase portrait and goes to the homoclinic wave for large periods. According with Theorem 1.1 and Lemma 3.1 in \cite{alvesnatali}, we see that $\mathcal{L}$ has only one negative eigenvalue which is simple and therefore, by Theorem $\ref{mainT}$ we have that all positive and periodic waves are transversally (spectrally) unstable.\\\\
\noindent \textbf{b) Periodic waves that change their sign.} Let $p$ be a positive even integer. By the arguments in Section 2, we see that equation $(\ref{ode})$ has two symmetric equilibrium points $(\pm((p+1)c)^{1/p},0)$ and, consequently, positive and negative periodic waves appear in the corresponding phase-portrait. Since both of them converges to the corresponding homoclinic waves for large periods, we obtain two symmetric solitary waves (which are positive and negative smooth functions). Turning around both solitary waves, we obtain periodic waves that change their sign with the zero mean property. Again by Theorem 1.1 and Lemma 3.1 in \cite{alvesnatali}, we deduce that $n(\mathcal{L})=2$ and thus, we can not decide if $n(R(0))$ is one or two using directly Proposition $\ref{propLR}$. In order to guarantee that $n(R(0))=1$, we need to use the Index Theorem in \cite[Theorem 5.3.2]{KP} to get a convenient formula to calculate $n(R(0))$ as $n(R(0))=n(\mathcal{L})-n_0-z_0$. As we have already mentioned in the proof of Proposition $\ref{propR0}$, quantities $z_0$ and $n_0$ are non-negative integers which are related with the quantity $(\mathcal{L}^{-1}1,1)_{L_{per}^2}$. If $(\mathcal{L}^{-1}1,1)_{L_{per}^2}<0$, we have $n_0=1$ and $z_0=0$, while $(\mathcal{L}^{-1}1,1)_{L_{per}^2}>0$ implies $n_0=z_0=0$ and for $(\mathcal{L}^{-1}1,1)_{L_{per}^2}=0$, we obtain $n_0=0$ and $z_0=1$. To apply Theorem $\ref{mainT}$, we need to prove that $(\mathcal{L}^{-1}1,1)_{L_{per}^2}\leq0$ because in this case, we have $n(R(0))=n(\mathcal{L})-n_0-z_0=1$. This fact is not so simple to obtain in general and in fact, it is possible to prove in a specific case that $(\mathcal{L}^{-1}1,1)_{L_{per}^2}\leq0$ for some values of $c$ depending on $L$ and $(\mathcal{L}^{-1}1,1)_{L_{per}^2}>0$ for other values of $c$. For instance, when $p=2$, we have periodic waves that change their sign as determined in \cite{AN1} and \cite{angulo}. Let $L_0>0$ be fixed. In \cite{AN1} the authors calculated a threshold value $c^*$ depending on the period $L_0$ where for $c\in (0,c^*]$, we have $(\mathcal{L}^{-1}1,1)_{L_{per}^2}\leq0$, so that $n(R(0))=1$ and the periodic wave that changes its $\varphi$ is transversally unstable. For $c> c^*$ we obtain, since $(\mathcal{L}^{-1}1,1)_{L_{per}^2}>0$, that $n(R(0))=2$ and Theorem $\ref{mainT}$ can not be applied. The value of $c^*$ can be explicitly determined as $c^*\approx\frac{56.277}{L_0^2}$. Next, we can obtain a similar result when $p=4$. In this case, we also have a threshold value $c^{\#}$ such that for $c\in(0,c^{\#}]$ the periodic wave that changes its $\varphi$ is transversally unstable since $(\mathcal{L}^{-1}1,1)_{L_{per}^2}\leq0$. The value of $c^{\#}$ can be explicitly computed as $c^{\#}\approx\frac{43.665}{L_0^2}$ (see \cite{NCA} for further details).

\begin{remark}
	If $p=2$ or $p=4$, we can determine the transversal (spectral) instability of the periodic wave $\varphi$ that changes its sign in the case $n(R(0))=2$. As we have already mentioned in Example b), Theorem 1.2 in \cite{nikolay} can not be applied since $P(k_0)$ is not a Fredholm operator with zero index. To prove the transversal instability in both cases, we need to use the spectral instabilities results in \cite{angulo} and \cite{NCA} for the modified and critical KdV, respectively. In fact, in both cases the spectral problem to study is similar but we need to consider $k=0$ in $(\ref{specprobmean})$, that is, we need to consider the spectral problem $\partial_xQ\mathcal{L}w=\lambda w$, where $w$ is a smooth real periodic function with the zero mean property. For $p=2$ and $c>c^*$, we obtain $(\mathcal{L}^{-1}1,1)_{L_{per}^2}>0$ and since $\frac{d}{dc}\int_0^{L}\varphi(x)^2dx>0$, we conclude by \cite[Theorem 4.2]{angulo} that the spectrum of $\partial_x Q\mathcal{L}$ has a complex eigenvalue with positive real part. By continuity, we then deduce, for $k_1>0$ small enough, that the spectrum of $\partial_x( Q\mathcal{L}+k_1^2)$ has the same property and the wave is transversally unstable according to the Definition $\ref{defistab1}$. In the case $p=4$ and $c>c^{\#}$, we have the same scenario as in the case $p=2$ by using \cite[Theorem 1.1-b)]{NCA}.
\end{remark}

\section{Remarks on the nonlinear instability.}

\indent In this section, we present some remarks concerning the nonlinear instability of the periodic wave $\varphi$ that is transversally (spectrally) unstable according to the Theorem $\ref{mainT}$. As we have already mentioned in the introduction, we need to start by assuming the following hypothesis in order to cover all possible cases of well-posedness results in the energy space $X=H^1(\mathbb{T}_L\times \mathbb{R})$ of the Cauchy problem associated to the equation $(\ref{ZK})$.

\begin{itemize}
\item[(H1)] The Cauchy problem associated to the gZK equation
	\begin{equation}\label{cauchygZK}
		\left\{\begin{array}{lll} u_t+u^pu_x+(\Delta u)_x=0,\ \ \ \ \ \ \ \
			(x,y,t)\in\mathbb{T}_L\times\mathbb{R}\times\mathbb{R}_{+},\\\\
			u(x,y,0)=u_0(x,y), \ \ \ \ \ \ \ \ \ \ \ \ 
			(x,y)\in\mathbb{T}_L\times\mathbb{R},
		\end{array}\right.
	\end{equation}
	is globally well-posed in $X$. In other words, if $u_0\in X$ there is a
	unique mild solution $u\in C([0,T];X)$, for all $T>0$. The
 data-solution map associated to the problem $(\ref{cauchygZK})$,
	\begin{equation}\label{data-sol}\begin{array}{lll}\Upsilon:X\rightarrow
			C([0,T];X)\\
			\ \ \ \ \  u_0\mapsto
			\Upsilon(u_0)=u_{u_0},\end{array}\end{equation} is smooth. In addition, associated to the equation $(\ref{cauchygZK})$, we
	have the following conserved quantities,
	\begin{equation}\label{conservada1}
		F(u)=\displaystyle\frac{1}{2}\int_{\mathbb{T}_L\times \mathbb{R}}u^2dxdy \ \ \ \ \mbox{and}
		\ \ \ E(u)=\frac{1}{2}\int_{\mathbb{T}_L\times \mathbb{R}}|\nabla u|^2-\frac{2}{(p+1)(p+2)}u^{p+2}dxdy.
	\end{equation}
\end{itemize}

\begin{remark}
	In general, local well-posedness results in  $X$ for the problem $(\ref{cauchygZK}])$ is obtained by using
	fixed point arguments. Global results in time are then obtained from the conserved quantity $E(u)=\frac{1}{2}\int_{\mathbb{T}_L\times \mathbb{R}}|\nabla u|^2-\frac{2}{(p+1)(p+2)}u^{p+2}dxdy$ and the standard Gagliardo-Nirenberg inequality. As far as we  know, this would happen when $1\leq p<2$. The smoothness of the
	data-solution map $\Upsilon$ given by $(\ref{data-sol})$ is determined by using the implicit function theorem.
\end{remark}

Some considerations deserve to be mentioned before
talking about nonlinear stability. Since equation $(\ref{ZK})$ is
invariant under translations, that is, if $u(x,t)$ is a solution then
$u(x+r,t)$ is also a solution for every $r\in\mathbb{R}$, we obtain
that the one-parameter group of unitary operators $\{S(r)\}_{r\in \mathbb
	R}$ defined by $S(r)f(\cdot)=f(\cdot+r)$  determines the
$\varphi$-orbit
$$
\Omega_{\varphi}=\{S(y)\varphi;\ y\in\mathbb{R}\}.
$$
Then, we say that $\Omega_{\varphi}$ is stable in the Hilbert space $X$ by
the flow of equation  $(\ref{ZK})$, if for all $\varepsilon>0$
there is $\delta>0$ such that if
$||u_0-\varphi||_{X}<\delta$ and $u(t)$ is
the global solution of  $(\ref{ZK})$ with initial data
$u(x,y,0)=u_0(x,y)$, then
$$
\displaystyle\inf_{r\in\mathbb{R}} ||u(t)
-S(r)\varphi||_{X}<\varepsilon,\ \mbox{for all}\ t\in \mathbb R.
$$
Otherwise, the orbit is said to be orbitally unstable in
$X$. More specifically, there exists $\eta > 0$ such that
for every $\delta > 0$, there exists $u_0^{\delta}$
 and a time $t^{\delta}>0$ such that
$||u_{0}^{\delta}-\varphi||_X<\delta$
and the solution $u^{\delta}$ of $(\ref{cauchygZK})$ with initial value $u_{0}^{\delta}$ satisfies $\inf_{r\in\mathbb{R}}||u^{\delta}(t^{\delta})-S(r)\varphi||\geq \eta$. This fact would be expected if the solution $u$ of the problem $(\ref{cauchygZK})$ has a blow-up in finite time. In our model, this fact is expected if $p\geq 2$ and for arbitrary initial data $u_0\in X$.

The following result links the nonlinear instability and the transversal (spectral) instability.

\begin{proposition}\label{henry1} Let $Y$ be a Banach space
	and $\mathcal{O}\subset Y$ an open set containing $0$. Suppose that
	$\mathcal{T}:\mathcal{O}\to Y$ is a map satisfying $\mathcal{T}(0)=0$. In addition, suppose that for some $q>1$, there exists a continuous
	linear operator $\mathcal S$ with spectral radius $r(\mathcal S)>1$
	such that
	$
	\|\mathcal{T}(w)-\mathcal S(w)\|_Y=O(\|w\|_Y^q)\;\;as\;\;w\to 0.
	$
	Then $0$ is unstable as a fixed point of $\mathcal{T}$.
\end{proposition}
\begin{proof}
	See \cite{henry}.
\end{proof}

\begin{obs}\label{henry2} By Proposition $\ref{henry1}$, we may estimate the direction in which points move away from $0$ under successive applications of 
$\mathcal{T}$. Choose any positive integer $m$ and any $\mu$ such that 
$0<\mu<\frac{1}{\sqrt{2}}$. Define the cone
$$\Sigma =\Sigma(m,\mu)=\{w\in Y;\ ||\mathcal{S}^m(w)||_Y\leq \mu [r(\mathcal{S})]^m||w||_Y\}.$$
If $0<q< (1/\sqrt{2} -\mu)/(\mu + r(\mathcal{S})^{-m}||\mathcal{S}^m||_Y)$, there exists $a_q >0$ such that: given any $a\in (0, a_q]$, arbitrarily small $\varepsilon_0 > 0$ and arbitrarily large $N_0 > 0$, there exist $N>N_0$  and $w\in Y$ such that 
$||w||\leq \varepsilon_0$,  $||\mathcal{T}^n(w)||_{Y}\leq a$ for $0<n\leq N$ and $d(\mathcal{T}^N(w),\Sigma)\geq qa$. In particular, we have $||\mathcal{T}^N(w)||_Y\geq qa$.
\end{obs}

\begin{coro}\label{corohen} Let $\mathcal{W}:\mathcal{O}\subset Y\to Y$ be a
	$C^{2}$ map defined in an open neighbourhood of a fixed point $\phi$. If there exists $\mu\in \sigma(\mathcal{W}'(\phi))$ such that
	$|\mu|>1$, then $\phi$ is an unstable fixed point of $\mathcal{W}$.
\end{coro}

\begin{proof}  For $w\in U\equiv\{v-\phi;\ v\in \mathcal{O}\}$, let us consider $\mathcal{T}(w)\equiv \mathcal{W}(w+\phi)-\phi$. Since $\phi$ is a fixed point of $\mathcal{W}$, we have
	$\mathcal{T}(0)=\mathcal{W}(\phi)-\phi=0$ with $1<|\mu|\leq r(\mathcal{W}'(\phi))$. By
	Taylor's formula, we obtain $ \mathcal{T}(w)=\mathcal{T}(0)+\mathcal{T}'(0)w+ O(\|w\|_Y^2)=\mathcal{W}'(\phi)
	(w)+O(\|w\|_Y^2)$ for all  $\|w\|_Y<<1$. Therefore, by Remark
	$\ref{henry2}$, there exists $\varepsilon_0>0$ such that for all $\eta>0$
 and a large enough $N_0\in \mathbb N$, there exists
	 $N>N_0$ and $v\in B(\phi;\eta)$ such that
	$\|\mathcal{W}^N(v)-\phi\|_Y\geq \varepsilon_0.$ This finishes the proof.
\end{proof}

\begin{prop}\label{insta}  The periodic solution $\varphi$ that is transversally (spectrally) unstable according to Theorem
	$\ref{mainT}$ is nonlinearly unstable.
\end{prop}

\begin{proof}
	By equation $(\ref{ZK})$, we see that
	$u(x-ct,y,t)$ is a solution of the equation
	\begin{equation}\label{newgZK}
		u_t-cu_{\xi}+u^pu_{\xi}+(\Delta u)_{\xi}=0,
	\end{equation}
	where $\xi=x-ct$. In addition, the periodic wave $\varphi$ obtained in Section 2 is now an
	equilibrium solution of the equation $(\ref{newgZK})$. Consider $G(u)=E(u)+cF(u)$,
	where $E$ and $F$ are  given by $(\ref{conservada1})$.
	We have that  $(\ref{newgZK})$ can be rewritten as
	\begin{equation}\label{hamiltgBBM2}
		u_t=JG'(u),
	\end{equation}
	where $J=\partial_x$.  Moreover, from $(\ref{hamiltgBBM2})$
	the linearized equation at the equilibrium point $\varphi$ is
	 $v_t=J(\mathcal{L}-\partial_y^2)v$, where $\mathcal{L}$ is the linear operator given by $(\ref{operator})$. Let us consider $v(x,y,t)=e^{iky}w(x,t)$ in the linearized equation to obtain $w_t=J(\mathcal{L}+k^2)w$.

	Define $\mathcal{W}:X\rightarrow X$ as $\mathcal{W}(u_0)=u_{u_0}(1)$,
	where $u_{u_0}(t)$ is the solution of $(\ref{newgZK})$ with initial data $u(x,y,0)=u_0(x,y)$ at $t=1$.
	For each $T>0$, function $\Upsilon:X\rightarrow C([0,T];X)$ is the
	data-solution map related to the equation
	$(\ref{newgZK})$ and by assumption (H1),
	$\Upsilon$ is smooth. Again by (H1), the uniqueness of solutions for the Cauchy problem $(\ref{cauchygZK})$ gives us that $\mathcal{W}(\varphi)=\varphi$
	and $\mathcal{W}$ is a $C^2$ map
	defined in a neighbourhood of $\varphi$ (this fact follows
	from the translation in $x$ as a linear continuous map defined $X$).
	Moreover, for $h(x,y)=e^{iky}g(x)\in X$ we have $\mathcal{W}'(\varphi)h=w_h(1)$, where
	$w_h(1)$ is the solution of the linear initial value problem
	\begin{equation}\label{linear1}\left\{\begin{array}{lll}
			w_t=J(\mathcal{L}+k^2)w\\
			w(0)=h,
		\end{array}\right.
	\end{equation}
	evaluated at $t=1$. Then, using Theorem $\ref{mainT}$, we obtain the existence of $\nu>0$, $k\neq0$ and $U\in X\backslash\{0\}$ such that
	$J(\mathcal{L}+k^2)U=\nu U$. Hence, for $w_U(t)=e^{\nu t} U$ and
	$\alpha= e^\nu$, we obtain $\mathcal{W}'(\varphi)U=w_U(1)=\alpha U$, that is,
	$\alpha\in\sigma(\mathcal{W}'(\varphi))$. By Corollary
	$\ref{corohen}$, we obtain the nonlinear instability in $Y$ of the periodic solution $\varphi$ that is transversally (spectrally) unstable according to the Theorem $\ref{mainT}$.
\end{proof}
	
	\section*{Acknowledgments}
F. Natali is partially supported by Funda\c c\~ao Arauc\'aria/Brazil (grant 002/2017) and CNPq/Brazil (grant 303907/2021-5).

\end{document}